\documentclass[preprint,review,sort&compress,3p]{elsarticle}

\usepackage{adjustbox}
\usepackage{graphicx}
\usepackage{enumerate}
\usepackage{amscd}
\usepackage{mathrsfs}
\usepackage{amssymb}
\usepackage{pifont}
\usepackage{amsmath}
\usepackage{amsthm}  
\usepackage{yhmath}
\usepackage{amsfonts}
\usepackage{booktabs}
\usepackage{makecell}
\usepackage{graphicx}
\usepackage{subcaption} 
\usepackage{epstopdf}
\usepackage{float}
\usepackage{longtable}
\usepackage{multirow}
\usepackage{bm}
\usepackage{algorithm}
\usepackage{algorithmicx}
\usepackage{dsfont}
\usepackage{algpseudocode}
\usepackage{diagbox}
\usepackage{epsfig}

\usepackage[
ragged
]{sidecap}   
\sidecaptionvpos{figure}{t} 
\usepackage{caption}
\usepackage{colortbl}
\usepackage{array, boldline, makecell, booktabs}
\usepackage{amsthm}
\usepackage{amsmath}
\usepackage{float}
\usepackage[colorlinks,
linkcolor=blue,
anchorcolor=blue,
citecolor=blue]{hyperref}

\usepackage{tikz}
\newcommand*{\circled}[1]{\lower.7ex\hbox{\tikz\draw (0pt, 0pt)%
	circle (.5em) node {\makebox[1em][c]{\small #1}};}}
	\newtheorem{thm}{Theorem }[section]

	\newtheorem{lem}{Lemma }[section]
	\newtheorem{dnt}{Definition}[section]
	\newtheorem{rk}{Remark }[section]
	
	\newtheorem{exm}{Example}[section]
	
	\journal{}
    \makeatletter
    \def\ps@pprintTitle{%
       \let\@oddhead\@empty
       \let\@evenhead\@empty
       \def\@oddfoot{\centerline{\thepage}}%
       \let\@evenfoot\@oddfoot}
    \makeatother
	\begin{document}

\begin{frontmatter}
	

	\title{{  Error Analysis of Parameter Prediction via Gaussian Process Regression and Its Application to Weighted Jacobi Iteration}}
	
	\author{Tiantian Sun$^{a}$, Juan Zhang$^{a,}$\footnote{Corresponding author.~~Email: zhangjuan@xtu.edu.cn  (J.Zhang)}\\
	$^a$ {\it Key Laboratory of Intelligent Computing and Information Processing of Ministry of Education, \\
		Key Laboratory for Computation and Simulation in Science and Engineering ,\\
		School of Mathematics and Computational Science, Xiangtan University, Xiangtan, Hunan,{\rm 411105},PR China}
	}
	\begin{abstract}
		
In this paper, we introduce a novel theoretical framework for Gaussian process regression error analysis, leveraging a function-space decomposition. Based on this framework, we develop a weighted Jacobi iterative method that utilizes Gaussian process regression for parameter prediction and provide a corresponding convergence analysis. Moreover, the convergence conditions are designed to be compatible with other error bounds, enabling a more general analysis. Experimental results show that the parameters predicted based on Gaussian process regression significantly accelerate the convergence speed of Jacobi iterations.
		
	\end{abstract}
	
\end{frontmatter}


\section{Introduction} \label{sec:1}

{  

The Jacobi iteration is a classical stationary iteration for solving linear systems 
$$Ax=b.$$ 
The weighted Jacobi iteration is a parameterized version of the Jacobi iteration, where a relaxation factor is introduced to potentially accelerate convergence~\cite{Y-2003}.
Furthermore, Anderson extrapolation has been applied to accelerate the Jacobi iteration, leading to the Anderson Jacobi and alternating Anderson Jacobi iterations~\cite{P-P-J-2016}. 
For finite-difference approximation of elliptic equations on large grids, \cite{X-R-2014} proposes a schedule of over- and under-relaxations of the Jacobi iteration. For solving symmetric positive semidefinite linear systems, \cite{L-Q-K-H-2025} integrates Nesterov's acceleration technique with the Jacobi and weighted Jacobi methods. This method improves convergence while maintaining the methods' inherent suitability for parallel implementation.

For these parameterized iterations, convergence is governed by the spectral radius of the corresponding iteration matrix, and the relaxation parameter that minimizes this spectral radius depends on the eigenvalue spectrum of $A$. 
However, for large scale linear systems, it is generally difficult to compute the spectral radius of $A$, which makes the determination of optimal parameters challenging in practice.
}

Notably, for numerical algorithms where parameter selection is challenging, Gaussian process  regression (GPR) provides an effective data-driven strategy for parameter prediction. GPR has been widely applied in many fields, including machine learning, statistical modeling, optimization, and engineering~\cite{C-C-2006, R-M-L-A-D-2014, W-P-W-2011, M-D-C-2013}. Recently, GPR has been employed for parameter optimization in numerical methods, including splitting parameter prediction in the general alternating-direction implicit framework,  Matrix splitting iteration methods, solution approximation in nonlinear PDEs, and parameter identification in inverse problems~\cite{F-D-S-M-2013,J-S-Z-2022, K-J-Q-2023, Y-B-H-A-2021}. This not only saves the time required for parameter computation but also significantly enhances the algorithm's performance. However, evaluating the accuracy of these predicted parameters to validate the effectiveness of the parameter selection approach remains a critical concern.

For many methods based on GPR, probabilistic uniform error bounds have been extensively studied \cite{T-T-F-2023}. Under the weak assumption of Lipschitz continuity of the covariance kernel and the unknown function, a directly computable probabilistic uniform error bound has been established in \cite{A-J-S-2019}. By constructing a confidence region in the hyperparameter space, \cite{A-A-S-2022} provides a probabilistic upper bound for the model error of a Gaussian process with arbitrary hyperparameters. Similarly, by bounding the information gain for kernel classes, the  regret bounds for Gaussian process optimization have been established in \cite{N-A-S-M-2012}. However, it remains challenging to incorporate such probabilistic error bounds into the convergence analysis of numerical algorithms.
Building upon the low-complexity assumptions in concentration inequalities and reproducing kernel Hilbert spaces (RKHS), \cite{R-L-M-2025} derives tighter probabilistic and deterministic error bounds that improve upon the foundational work of \cite{K-A-M-T-D-2022}.

Due to the fact that noise-free GPR, a Bayesian nonparametric method also known as Kriging or Wiener-Kolmogorov prediction, is mathematically equivalent to an interpolation problem \cite{M-P-D-B-2018,S-J-P-E-2023}, this property enables us to derive a deterministic error bound for it.
Error bounds for the predicted function in GPR under kernel misspecification have been established using an interpolation-based approach  \cite{W-B-2022,R-W-2020}.
To address both epistemic and computational misspecification, \cite{D-R-2025} employs an approach based on interpolation to analyze the approximation error of predicted function $f^*$ by its Gaussian process interpolant when the kernels are misspecified.

With the widespread application of GPR in numerical computations, accurately quantifying the impact of its prediction error on algorithm performance has become crucial.
{  Here, starting from the weighted Jacobi iteration (WJI), we use GPR as the strategy for selecting the optimal weight and provide a convergence analysis of the iterative method under the parameter predicted by GPR.}
In this paper, we propose a novel theoretical framework based on function space decomposition, which presents a new proof method for error bound by representing the prediction function as a combination of finite point interpolation and global correction term. 
This universal framework can be naturally extended to different kernel functions and various noise distribution scenarios. 
Furthermore, we have developed an application of this framework to accelerate weighted Jacobi iteration, providing a convergence analysis with general and flexible conditions that can be readily combined with other GPR error bounds.
Numerical results demonstrate that the parameter obtained by GPR can accelerate the convergence of the {  WJI}.  Meanwhile, by comparing the prediction error bounds of GPR with different kernel functions, it is found that when a proper kernel function is selected, the prediction error will decrease, thereby accelerating the convergence speed of the algorithm.

The rest of the paper is organized as follows. In Section \ref{sec:error_bounds}, we establish the new proof method for error bound. In Section \ref{WJI-GPR}, we propose the  WJI based on GPR and give the corresponding convergence analysis.  Numerical experiments are performed in Section \ref{Experiments} to validate the proposed method and verify the feasibility of applying GPR to parameter prediction with this error boundary. We conclude with a summarizing discussion in Section \ref{Conclusion}.

\section{Error bounds for Gaussian process regression}\label{sec:error_bounds}
\subsection{Notation}
We start with introducing the basic definitions and concepts.
Let $\Omega \subseteq \mathbb{R}^d, d\in \mathbb{N}$ and denote $C(\Omega)$ as the function space consisting of functions that are one times continuously differentiable on $\Omega$.

\begin{dnt}[\cite{P-A-R-M-2004}] 
	A continuous kernel $\kappa: \Omega \times \Omega \rightarrow \mathbb{C}$ is called positive definite on $\Omega \subseteq \mathbb{R}^d$ if for all $N \in \mathbb{N}$, all pairwise distinct $X = \{x_1,\dots,x_N\}\subseteq \Omega$, and all $\alpha\in \mathbb{C}^N \backslash\{0\}$ we have 
	$$
	\sum\limits_{i=1}^N \sum\limits_{j=1}^N \alpha_i \bar{\alpha_j} \kappa(x_i,x_j) >0
	$$
\end{dnt}

\begin{dnt}[\cite{P-A-R-M-2004}] Let $\mathcal{H}$ be a real Hilbert space of functions $f: \Omega\rightarrow\mathbb{R}$ with inner product $\langle\cdot,\cdot\rangle_\mathcal{H}$. A function $\kappa: \Omega\times \Omega \rightarrow \mathbb{R}$ is called a reproducing kernel for $\mathcal{H}$ if \\
	(1): $\kappa(\cdot,y)\in \mathcal{H}$ for all $y\in \Omega$, \\
	(2): $f(y) = \langle f,\kappa(\cdot,y)\rangle_\mathcal{H}$ for all $f \in \mathcal{H}$ and all $y\in \Omega$. 
\end{dnt}

\begin{dnt}[\cite{R-1999}]
	\label{native_space}
	If a symmetric positive definite function $\kappa:\Omega \times \Omega \rightarrow \mathbb{R}$	is the reproducing kernel of a real Hilbert space { $\mathcal{H}$} of real-valued functions on $\Omega$, then { $\mathcal{H}$} is the native space for $\kappa$.
\end{dnt}

{ 
Consequently, the native space $\mathcal{H}$ is precisely the Hilbert space associated with the kernel $\kappa$. It is the unique Reproducing Kernel Hilbert Space (RKHS) associated with $\kappa$, 
it can be equivalently defined as the completion of $\operatorname{span}\{\kappa(\cdot, y): y \in \Omega\}$ with respect to the norm induced by the inner product $\langle \cdot, \cdot \rangle_{\mathcal{H}}$.
}
 Furthermore we have $\Vert f\Vert_\mathcal{H} = \sqrt{\langle f,f \rangle_\mathcal{H}} = \sqrt{\sum\limits_{i=1}^N\sum\limits_{j = 1}^N \alpha_i \bar{\alpha_j} \kappa(x_i,x_j) }$ which implies both the magnitude and smoothness of  $f\in \mathcal{H}$, decreases as $f$ becomes smoother and increases as $f$ becomes less smooth.

\subsection{Proof Framework for the Error Bound of Gaussian Process Regression}
We start from noise-free GPR and suppose that training set $(x_i,y_i)$ satisfy the following model without noise
$$
y_i = f(x_i),~i = 1,2, \dots, N,
$$
where $x_i\in \Omega\subset\mathbb{R}^d$. For notational convenience, we denote $X_n = \{x_1,x_2,\dots,x_n\}$. To Without loss of generality, assume true function $f$ follows a zero mean Gaussian process~\cite{C-C-2006}, i.e. $Y = (f(x_1),f(x_2), \dots, f(x_N))^T \sim \mathcal{N}(0,\Sigma)$
, where
$$
\Sigma = 
\begin{pmatrix}
	\kappa(x_1,x_1) & \cdots & \kappa(x_1,x_N) \\
	\vdots			& \ddots & \vdots 	     \\
	\kappa(x_N,x_1) & \cdots & \kappa(x_N,x_N) \\
\end{pmatrix},
$$
$\kappa(\cdot,\cdot)$ is covariance function. 
Given new input data $\boldsymbol{z}$, for simplicity, we denote $K_{XX} = (\kappa(x_i,x_j))$ and $K_X(\boldsymbol{z}) = (\kappa(x_1,\boldsymbol{z}), \kappa(x_2,\boldsymbol{z}), \dots, \kappa(x_N,\boldsymbol{z}))$, then 
$$
\begin{pmatrix}
	Y\\f^*
\end{pmatrix}
\sim
\mathcal{N}\begin{pmatrix}
	\begin{bmatrix}
		0 \\ 0
	\end{bmatrix}, 
	\begin{bmatrix}
		K_{XX} & K_X(\boldsymbol{z})^T \\ K_X(\boldsymbol{z}) &\kappa(\boldsymbol{z},\boldsymbol{z})
	\end{bmatrix}
\end{pmatrix}. 
$$

{  If the kernel function $\kappa$ is positive definite and the training inputs $X$ are distinct, then the kernel matrix $K_{XX}$ is invertible. In practice, a small regularization term $\eta I$ (with $0<\eta\ll 1$) is often added to $K_{XX}$ to improve numerical stability and avoid ill-conditioning, ensuring that $K_{XX} + \eta I$ is invertible. In practice, we set $\eta = 10^{-4}$. For simplicity in the analysis, we assume $K_{XX}$ is invertible.
}

{  Then} the conditional expectation and variance of $f(\boldsymbol{z})$ are given by 
\begin{equation}
	\label{eq:mean}
	\mathbb{E} ( f(\boldsymbol{z}) \vert Y ) = K_X(\boldsymbol{z}) K_{XX}^{-1} Y,
\end{equation}
\begin{equation}
	\label{eq:variance}
	\text{Var}( f(\boldsymbol{z}) \vert Y ) = \kappa(\boldsymbol{z},\boldsymbol{z})-K_X(\boldsymbol{z}) K_{XX}^{-1} K_X(\boldsymbol{z}). 
\end{equation}

Notice that (\ref{eq:mean}) is a natural predict function of $f(\boldsymbol{z})$ which is denoted as  $f^*(\boldsymbol{z})$, i.e. 
$$
f^*(\boldsymbol{z}) = K_X(\boldsymbol{z}) K_{XX}^{-1} Y.
$$

To get the error bound between true function $f$ and the predicted function $f_*$, 
the function $f(\boldsymbol{z})$ can be first expressed as the sum of its projection onto the subspace spanned by $\{p_1,p_2,\dots, p_q\}$ and a global correction term in the $\mathcal{H}$, where $p_i$, $i = 1,\dots,q$ are linearly independent functions.

\begin{thm}
	\label{thm_Xi}
	Assume $f(\boldsymbol{z}) = \sum\limits_{j=1}^{\infty} \beta_j \kappa (z,x_j) \in \mathcal{N}_\mathcal{H}$, where $\kappa(z,x)$ is a reproducing kernel associated with a Hilbert space $\mathcal{H}$. Let  $p_1,p_2,\dots,p_q$ be a set of linearly independent functions in $C(\Omega)$. 
	For a subset $\Xi = \{\xi_1,\xi_2,\dots,\xi_q \} \subseteq \Omega$, the function $f(z)$ can be rewritten as:
	\begin{equation}
		\label{pro_f_z}
		f(\boldsymbol{z}) = \sum\limits_{i = 1}^q f(\xi_i) p_i(\boldsymbol{z}) +\left\langle f,\kappa(\cdot,\boldsymbol{z})-\sum\limits_{i= 1}^q p_i(\boldsymbol{z})\kappa(\cdot,\xi_i) \right\rangle_\mathcal{H}.
	\end{equation}
	
\end{thm}
\begin{proof}
	According to the definition of the Hilbert space inner product, we have
	$$
	\begin{aligned}
		&\left\langle f,\kappa(\cdot,\boldsymbol{z})-\sum\limits_{i= 1}^q p_i(\boldsymbol{z})\kappa(\cdot,\xi_i) 	\right\rangle_\mathcal{H} 
		= \left\langle \sum\limits_{j = 1}^{\infty}\beta_j \kappa(\cdot,x_j) ,\kappa(\cdot,\boldsymbol{z})-\sum\limits_{i= 1}^q p_i(\boldsymbol{z})\kappa(\cdot,\xi_i) \right\rangle_\mathcal{H} \\
		&= \sum\limits_{j = 1}^{\infty} \beta_j \kappa(\boldsymbol{z},x_j) - \sum\limits_{i = 1}^q \sum\limits_{j = 1}^{\infty} \beta_j p_i(\boldsymbol{z}) \kappa(\xi_i,x_j)
		= f(\boldsymbol{z}) - \sum\limits_{i = 1}^q \sum\limits_{j = 1}^{\infty} \beta_j p_i(\boldsymbol{z}) \kappa(\xi_i,x_j)\\
		&= f(\boldsymbol{z}) - \sum\limits_{i = 1}^q f(\xi_i) p_i(x),
	\end{aligned}
	$$
	then (\ref{pro_f_z}) holds.
\end{proof}

This decomposition separates the prediction function into two structurally interpretable components, i.e., the projection term that captures low-dimensional patterns through the specified basis functions $\{p_i\}_{i = 1}^q$ and the RKHS-orthogonal correction term enables capturing more latent information in the dataset based on the kernel's properties.
The projection term represents the prediction function's behavior constrained to the finite-dimensional subspace $\text{span} \{p_1,p_2,\dots, p_q\}$, while the correction term characterizes the correction required by the RKHS norm to maintain the GP's nonparametric flexibility.

\begin{lem}
	\label{lem_g}
	Let $\kappa(x, x_1), \kappa(x, x_2), \dots, \kappa(x, x_N)$ be a set of linearly independent functions, and let $p_1(x), \dots, p_q(x)$ be another set of linearly independent functions. Let $A = (a_{ij})$ be an invertible matrix, and define the functions $g_i(x)$ as:
	$$
	g_i(x) = \sum_{j=1}^N \sum_{l=1}^N \kappa(x, x_j) a_{jl} p_i(x), \quad i = 1, \dots, q,
	$$
	then, $g_1(x), g_2(x), \dots, g_q(x)$ are also linearly independent functions.

\end{lem}
\begin{proof}
	Let $C = (c_1,c_2,\dots,c_q)^T$, such that
	\begin{equation}
		\label{cg}
		c_1 g_1(x)+c_2 g_2(x)+\cdots+c_q g_q(x) = 0.
	\end{equation}
	Let $P = (p_i(x_j))\in \mathbb{R}^{q\times N}$. Combining the expression of $g_i(x)$,  (\ref{cg}) can be rewritten as
	\begin{equation}
		\label{cg-1}
		K_X(x) K_{XX}^{-1} P^T C = 0. 
	\end{equation}
	Since $\kappa(x, x_1), \kappa(x, x_2), \dots, \kappa(x, x_N)$ are linearly independent, and $K_{XX}^{-1}$ is invertible, we can rewrite (\ref{cg-1}) as
	$$
	P^TC = 0. 
	$$
	Further, $p_1(x),p_2(x),\dots,p_q(x)$ are linearly independent, then $C = 0$, thus $g_1(x),g_2(x),\dots,g_q(x)$ are linearly independent functions.
\end{proof}

From Lemma \ref{lem_g}, we know that any invertible transformation maintains the linear independence of the combined function system formed by the kernel functions $\kappa(x,x_j), j = 1,\dots,N$ and the basis $p_i(x)$, $i = 1,2, \dots, q$.

\begin{thm}
	\label{thm_GPR_nf}
	Let $\Omega \subseteq \mathbb{R}^d$ be open. Suppose that $\kappa$ is a positive definite kernel on $\Omega$. Then for every $\boldsymbol{z}\in \Omega$ the error between true function $f\in \mathcal{N}_\mathcal{H} (\Omega)$ and its prediction $f^*$ from GPR can be bounded by 
	\begin{equation}
		\vert f(\boldsymbol{z})-f^*(\boldsymbol{z})\vert^2 \le C\cdot Var(f^*(\boldsymbol{z})),
	\end{equation}
	where $C$ is an nonnegative constants, and $Var(f^*(\boldsymbol{z}))$ is the variance of $f^*(\boldsymbol{z})$.
\end{thm}
\begin{proof}
	Let $p_1(x), p_2(x),\dots,p_q(x)$ be a set of linearly independent functions, $\mathcal{P}$ is a space spanned by $\{p_1,p_2,\dots,p_q\}$, and choose a subset $\Xi = \{\xi_1,\xi_2,\dots,\xi_q \} \subseteq \Omega$ . Assume $\dim(\mathcal{P})=q \le N$, and let $\{p_i(\boldsymbol{z})\}_{i = 1}^q$ is a basis of $\mathcal{P}$.
	Then, according to Theorem \ref{thm_Xi}, $f(\boldsymbol{z})$ can be rewritten as 
	\begin{equation}
		\label{f_z}
		f(\boldsymbol{z}) = \sum\limits_{i = 1}^q f(\xi_i) p_i(\boldsymbol{z}) +\left\langle f,\kappa(\cdot,\boldsymbol{z})-\sum\limits_{i= 1}^q p_i(\boldsymbol{z})\kappa(\cdot,\xi_i) \right\rangle_\mathcal{H}.
	\end{equation}

	Let $g_i(x) :=\sum\limits_{j =1}^N\sum\limits_{k=1}^N \kappa(x,x_j)a_{jk} p_i(x_k)   = \sum\limits_{j=1}^N \kappa(x,x_j) g_i(x_j)$, where $g_i(x_j) = \sum\limits_{k = 1}^N a_{jk}p_i(x_k)$. 
	Let $U^*(\boldsymbol{z}) =  K_X(\boldsymbol{z}) K_{XX}^{-1}$, then $\bar{f^*}(\boldsymbol{z})  = U^*(\boldsymbol{z})Y$. {  There is}

	$$
	\begin{aligned}
		f^*(\boldsymbol{z}) &= U^*(\boldsymbol{z})Y =\sum\limits_{j = 1}^N u_j^*(\boldsymbol{z})f(x_j)\\
		&=\sum\limits_{j = 1}^N u_j^*(\boldsymbol{z}) \left(\sum\limits_{i = 1}^q f(\xi_i) p_i(x_j) +\left\langle f,\kappa(\cdot,x_j)-\sum\limits_{i= 1}^q p_i(x_j)\kappa(\cdot,\xi_i) \right\rangle_\mathcal{H} \right)\\
		&=\sum\limits_{i = 1}^q \sum\limits_{j = 1}^N u_j^*(\boldsymbol{z}) f(\xi_i) p_i(x_j)+ \left\langle f, \sum\limits_{j = 1}^N u_j^*(\boldsymbol{z})\kappa(\cdot,x_j)- \sum\limits_{i= 1}^q\sum\limits_{j = 1}^N u_j^*(\boldsymbol{z})p_i(x_j)\kappa(\cdot,\xi_i) \right\rangle_\mathcal{H}.
	\end{aligned}
	$$
	
	Since $U^*(\boldsymbol{z}) =\begin{bmatrix}u_1^*(\boldsymbol{z}) & u_2^*(\boldsymbol{z}) &\cdots &u_N^*(\boldsymbol{z}) \end{bmatrix}= K_X (\boldsymbol{z}) (K_{XX})^{-1}$,  $\sum\limits_{i = 1}^q \sum\limits_{j = 1}^N u_j^*(\boldsymbol{z}) f(\xi_i) p_i(x_j) $ can be rewritten as $ K_X(\boldsymbol{z}) K_{XX}^{-1} P^T Y_{\xi}$, where $P = (p_i(x_j))\in \mathbb{R}^{q\times N}$, $Y_{\xi} =(f(\xi_1),f(\xi_2), \dots, f(\xi_q) )^T $. Let $(g_1(\boldsymbol{z}),\dots,g_q(\boldsymbol{z})) =K_X(\boldsymbol{z}) K_{XX}^{-1} P^T $, according to Lemma \ref{lem_g}, $g_1(\boldsymbol{z}), g_2(\boldsymbol{z}),\dots,g_q(\boldsymbol{z})$ are linearly independent. Obviously, we have 
	$$
	\sum\limits_{i = 1}^q \sum\limits_{j = 1}^N u_j^*(\boldsymbol{z})  p_i(\boldsymbol{x}_j) = \sum\limits_{i = 1}^q  g_i(\boldsymbol{z} ) .
	$$ 
	
	Then $f^*(\boldsymbol{z})$ can be rewritten as 
	\begin{equation}
		f^*(\boldsymbol{z}) = \sum\limits_{i = 1}^q  g_i(\boldsymbol{z}) f(\xi_i) +\left\langle f, \sum\limits_{j = 1}^N u_j^*(\boldsymbol{z})\kappa(\cdot,\boldsymbol{x}_j)- \sum\limits_{i= 1}^q p_i(\boldsymbol{z})\kappa(\cdot,\xi_i)\right\rangle_\mathcal{H}. 
	\end{equation}
	
	In the other hand, the functions $g_1(x), g_2(x),\dots,g_q(x)$ are also linearly independent, then 
	\begin{equation}
		f(\boldsymbol{z}) = \sum\limits_{i = 1}^q f(\xi_i) g_i(\boldsymbol{z}) +\left\langle f,\kappa(\cdot,\boldsymbol{z})-\sum\limits_{i= 1}^q g_i(\boldsymbol{z})\kappa(\cdot,\xi_i) \right\rangle_\mathcal{H}.
	\end{equation}
	
	Thus 
	$$
	\begin{aligned}
		&\vert f(\boldsymbol{z}) - f^*(\boldsymbol{z})\vert^2 =\left\vert \left \langle f,\kappa(\cdot,\boldsymbol{z}) - \sum\limits_{j = 1}^N u_j^*(\boldsymbol{z})\kappa(\cdot,\boldsymbol{x}_j) \right \rangle_\mathcal{H} \right\vert^2 \\
		&\le \Vert f\Vert_\mathcal{H}^2 \cdot \left\vert \left \langle \kappa(\cdot,\boldsymbol{z}) - \sum\limits_{j = 1}^N u_j^*(\boldsymbol{z})\kappa(\cdot,\boldsymbol{x}_j), \kappa(\cdot,\boldsymbol{z}) - \sum\limits_{j = 1}^N u_j^*(\boldsymbol{z})\kappa(\cdot,\boldsymbol{x}_j) \right \rangle_\mathcal{H} \right\vert\\
		&= \Vert f\Vert_\mathcal{H}^2 \cdot \left(\kappa(\boldsymbol{z},\boldsymbol{z})- K_X(\boldsymbol{z}) K_{XX}^{-1} K_X(\boldsymbol{z})\right) = C\cdot Var(f^*(\boldsymbol{z})), 
	\end{aligned}
	$$
	where $C =\Vert f\Vert_\mathcal{H}^2$. 
\end{proof}
This theorem admits a straightforward extension to the noisy observation case.

\begin{thm}
	Let $\Omega \subseteq \mathbb{R}^d$ be open. Suppose that $\kappa$ is a positive definite kernel on $\Omega$. Consider the noisy observations $y_i = f(x_i)+\epsilon_i$, where $\epsilon_i\sim\mathcal{N}(0,\sigma_n^2)$ is observational noise with mean zero and variance $\sigma_n^2$.
	Then for every $\boldsymbol{z}\in \Omega$, the predicted function obtained by GPR $f_{noise}^*(\boldsymbol{z})$  and  its variance are given by 
	\[f_{noise}^*(\boldsymbol{z}) = K_X(\boldsymbol{z}) (K_{XX}+\sigma_n^2 I)^{-1} Y,
	\]
	\[
	Var(f_{noise}^*(\boldsymbol{z})) = \kappa(\boldsymbol{z},\boldsymbol{z})-K_X(\boldsymbol{z}) (K_{XX}+\sigma_n^2 I)^{-1} K_X(\boldsymbol{z}).
	\]
	Then the error between true function $f\in \mathcal{N}_\mathcal{H} (\Omega)$ and its prediction $f_{noise}^*$ from GPR can be bounded by 
	\begin{equation}
		\vert f(\boldsymbol{z})-f_{noise}^*(\boldsymbol{z})\vert^2 \le C_{noise}\cdot Var(f_{noise}^*(\boldsymbol{z})),
	\end{equation}
	where $C_{noise}$ {  is an nonnegative constant.}
\end{thm}

\section{Weighted Jacobi Iteration with GPR}
\label{WJI-GPR}
Consider a linear systems of equations of the form
\begin{equation}
	\label{Ax=b}
	Ax = b,
\end{equation}
where $A \in \mathbb{C}^{N\times N}$ is the coefficient matrix, $x\in \mathbb{C}^{N\times 1}$ is the unknown vector and $b\in \mathbb{C}^{N\times 1}$ is the right-hand side vector. Here $\mathbb{C}$ denotes the set of complex numbers. The matrix $A$ can be split as
$$
A = L+D, 
$$
where $D=\text{diag}(A)$ { is the diagonal part} and $L = A-D$ contains the off-diagonal entries.

The weighted Jacobi iteration(WJI) \cite{Y-2003} for solving equation (\ref{Ax=b}) is defined by the following update:
$$
x_{t+1} = x_t+\omega D^{-1}(b-Ax_t), 
$$
where $t$ is $t$-th iteration. The iteration converges if and only if $\rho(I-\omega D^{-1}A)<1$.  Denote the optimal parameter as $\omega_{opt}$ and $\lambda_i, i=1, \dots, n$ are eigenvalues of $D^{-1}A$. If $A$ is symmetric and all eigenvalues of $D^{-1}A$ are real,  $\lambda_{\max}$ and $\lambda_{\min}$ have the same sign, there is 

\begin{equation}
	\rho(I-\omega D^{-1}A) = \max\limits_{i}\left\vert 1-\omega \lambda_i  \right\vert = \max\left\{ \left\vert 1-\omega\lambda_{\min} \right\vert, \left\vert 1-\omega\lambda_{\max} \right\vert \right\}. 
\end{equation}

When $\lambda_{\max}$ has a different sign from $\lambda_{\min}$, one can apply a shift to ensure all eigenvalues have the same sign.
When $\left\vert 1-w\lambda_{\min} \right\vert = \left\vert 1-w\lambda_{\max} \right\vert$, $w$ is taken as the optimal value, i.e. $w = w_{opt}$, where
$$
\omega_{opt} = \frac{2}{\lambda_{\max}+\lambda_{\min}}. 
$$

Since the exact values of $\lambda_{\max}$ and $\lambda_{\min}$ are generally unavailable, determining the optimal relaxation parameter for the weighted Jacobi iteration becomes challenging. 
To address this, we employ GPR to estimate a quasi-optimal parameter, which is then utilized in subsequent iterations. The { WJI based on GPR} is outlined in Algorithm \ref{WJI_GPR}.

\begin{algorithm}
	\caption{The WJI based on GPR}
	\label{alg1}
	\label{WJI_GPR}
	\begin{algorithmic}[1]
		\Require $A$, $D$, $l$, $x_0$ with $\Vert x_0 \Vert_2 = 1$ and $\omega$ obtained by GPR;
		\Ensure $x_{l}$;
		\For {$t = 0,1,\dots,l-1$}
		\State $x_{t+1} = x_t+\omega D^{-1}(b-Ax_t)$;
		\EndFor
	\end{algorithmic}
\end{algorithm}

Assume $\omega^*$ is obtained by GPR, and $f^*$ is the corresponding prediction function. According to Theorem \ref{thm_GPR_nf}, there is 
\begin{equation}
	\vert \omega^*-\omega_{opt}\vert \le C_\omega \cdot Var(f^*), 
\end{equation}
where $C_\omega>0$ is a constant. Then convergence analysis of the weighted Jacobi iteration with $\omega^*$ is given as follows. 

\begin{thm}
	\label{GPR_Jacobi_convergence}
	Suppose that $A\in \mathbb{R}^{n\times n}$ is symmetric and let $\lambda_{\min}$ and $\lambda_{\max}$ denote the minimum and maximum eigenvalues of $D^{-1}A$, respectively. If $\lambda_{\min}$ and $\lambda_{\max}$ have the same sign, then the weighted Jacobi iteration with GPR predicted parameter $\omega^*$ converges if and only if 
	\begin{equation}
		\label{convergence_1}
		C_\omega \cdot Var(f^*) \le \frac{2\lambda_{\min}}{\lambda_{\max}(\lambda_{\max}+\lambda_{\min})}, ~ \lambda_{\max}>\lambda_{\min}>0,
	\end{equation}
	\begin{equation}
		\label{convergence_2}
		C_\omega \cdot Var(f^*)\le -\frac{2\lambda_{\max}}{\lambda_{\min}(\lambda_{\max}+\lambda_{\min})}, ~ 0>\lambda_{\max}>\lambda_{\min}.
	\end{equation}
\end{thm}

\begin{proof}
	The convergence of the weighted Jacobi iteration is equivalent to the spectral radius condition
	\begin{equation}
		\label{convergence_eq1}
		\rho(I-\omega^* D^{-1}A)  <1.
	\end{equation}
	Since $A$ is symmetric matrix, there is
	$$\rho(I-\omega^* D^{-1}A) = \max\limits_{i}\left\vert 1-\omega^* \lambda_i  \right\vert = \max\left\{ \left\vert 1-\omega^*\lambda_{\min} \right\vert, \left\vert 1-\omega^*\lambda_{\max} \right\vert \right\}. $$
	For the term $\left\vert 1-\omega^*\lambda_{\min} \right\vert$, there has
	\begin{align*}
		\left\vert 1-w^*\lambda_{\min}  \right\vert & = \left\vert 1-(\omega^*-\omega_{opt}+\omega_{opt})\lambda_{\min} \right\vert \\
		& \le \left\vert \frac{\lambda_{\max}-\lambda_{\min}}{\lambda_{\max}+\lambda_{\min}}\right\vert + C_\omega \cdot Var(f^*)\left\vert \lambda_{\min} \right\vert. 
	\end{align*}
	
	It is similar to get $\left\vert 1-\omega^*\lambda_{\max}  \right\vert \le \left\vert \frac{\lambda_{\max}-\lambda_{\min}}{\lambda_{\max}+\lambda_{\min}} \right\vert + C_\omega \cdot Var(f^*)\left\vert \lambda_{\min} \right\vert$. Then we can rewritten $\rho(I-\omega^* D^{-1} A)$ as follows:
	\begin{equation}
		\label{The2-2_rho}
		\rho(I-\omega^* D^{-1} A) = \left\vert \frac{\lambda_{\max}-\lambda_{\min}}{\lambda_{\max}+\lambda_{\min}}\right\vert+C_\omega \cdot Var(f^*) \max\{\vert \lambda_{\min}\vert, \vert \lambda_{\max}\vert\}
	\end{equation} 
	
	If $\lambda_{\max}>\lambda_{\min}>0$, there are $\left\vert \frac{\lambda_{\max}-\lambda_{\min}}{\lambda_{\max}+\lambda_{\min}} \right\vert = \frac{\lambda_{\max}-\lambda_{\min}}{\lambda_{\max}+\lambda_{\min}}$ and $\max\{\vert \lambda_{\min}\vert, \vert \lambda_{\max}\vert\} = \lambda_{\max} $. By substituting (\ref{convergence_1}) into (\ref{The2-2_rho}), we have
	\begin{align*}
		&\rho(I-\omega^* D^{-1} A)  = \frac{\lambda_{\max}-\lambda_{\min}}{\lambda_{\max}+\lambda_{\min}}+C_\omega \cdot Var(f^*) \lambda_{\max}\\
		&<\frac{\lambda_{\max}-\lambda_{\min}}{\lambda_{\max}+\lambda_{\min}}+\frac{2\lambda_{\min}\lambda_{\max}}{\lambda_{\max}(\lambda_{\max}+\lambda_{\min})}<1.
	\end{align*}
	
	If $0>\lambda_{\max}>\lambda_{\min}$, there are $\left\vert\frac{\lambda_{\max}-\lambda_{\min}}{\lambda_{\max}+\lambda_{\min}} \right\vert = -\frac{\lambda_{\max}-\lambda_{\min}}{\lambda_{\max}+\lambda_{\min}}$ and $\max\{\vert \lambda_{\min}\vert, \vert \lambda_{\max}\vert\} =-\lambda_{\min} $.  By substituting (\ref{convergence_2}) into (\ref{The2-2_rho}), we have
	\begin{align*}
		&\rho(I-\omega^* D^{-1} A)  = -\frac{\lambda_{\max}-\lambda_{\min}}{\lambda_{\max}+\lambda_{\min}}+C_\omega \cdot Var(f^*)(- \lambda_{\min})\\
		&<-\frac{\lambda_{\max}-\lambda_{\min}}{\lambda_{\max}+\lambda_{\min}}+\frac{2\lambda_{\min}\lambda_{\max}}{\lambda_{\min}(\lambda_{\max}+\lambda_{\min})}<1.
	\end{align*}
\end{proof}

\begin{rk}
	We provide some remarks on Theorem \ref{GPR_Jacobi_convergence} as follows:
	\begin{itemize}
		\item When \( Var(f^*) = 0 \), it follows that \( \omega^* = \omega_{opt} \), and the spectral radius achieves its minimum value \( \rho(I - \omega^* D^{-1} A) = \left| \frac{\lambda_{\max} - \lambda_{\min}}{\lambda_{\max} + \lambda_{\min}} \right| \), yielding the optimal convergence rate \( R_{opt} = -\ln\left( \left| \frac{\lambda_{\max} - \lambda_{\min}}{\lambda_{\max} + \lambda_{\min}} \right| \right) \).
		\item The convergence rate \( R = -\ln(\rho(I - \omega^* D^{-1} A)) \) increases as \( Var(f^*) \) decreases. Specifically, from the proof of Theorem \ref{GPR_Jacobi_convergence}, \( \rho(I - \omega^* D^{-1} A) \leq \left| \frac{\lambda_{\max} - \lambda_{\min}}{\lambda_{\max} + \lambda_{\min}} \right| + C_\omega Var(f^*) \max \{ |\lambda_{\min}|, |\lambda_{\max}| \} \). A smaller \( Var(f^*) \) leads to a smaller spectral radius, thereby accelerates convergence. To minimize \( Var(f^*) \),  the kernel function \( \kappa \) in the GPR model can be optimized, as indicated by (\ref{eq:variance}).
	\end{itemize}
\end{rk}

Furthermore, we derive a convergence condition on the approximation error $\Vert f-f^*\Vert_{L^\infty(\Omega)}$. It is obviously that 
\[\vert \omega^*-\omega_{opt}\vert \le  \Vert f-f^*\Vert_{L^\infty(\Omega)}. \]
Following a similar proof approach as in Theorem \ref{GPR_Jacobi_convergence}, we can obtain the Lemma \ref{GPR_Jacobi_convergence2}. This result can be directly combined with existing error bounds for GPR, thereby enabling a more comprehensive analysis of the GPR  for the weighted Jacobi iteration.

\begin{lem}
	\label{GPR_Jacobi_convergence2}
	Suppose that $A\in \mathbb{R}^{n\times n}$ is symmetric and $\lambda_{\min}$ and $\lambda_{\max}$ have the same sign. Then the weighted Jacobi iteration based on GPR predicted parameter $\omega^*$ converges if and only if 
	\begin{equation}
		\label{convergence_2_1}
		\Vert f-f^*\Vert_{L^\infty(\Omega)} \le \frac{2\lambda_{\min}}{\lambda_{\max}(\lambda_{\max}+\lambda_{\min})}, ~ \lambda_{\max}>\lambda_{\min}>0,
	\end{equation}
	\begin{equation}
		\label{convergence_2_2}
		\Vert f-f^*\Vert_{L^\infty(\Omega)} \le -\frac{2\lambda_{\max}}{\lambda_{\min}(\lambda_{\max}+\lambda_{\min})}, ~ 0>\lambda_{\max}>\lambda_{\min}.
	\end{equation}
\end{lem}

To quantify the parameter predicted by GPR under both epistemic and computational misspecification, we leverage the $L^\infty$ interpolation error bound established in \cite{D-R-2025} to our setting.

For $s < \frac{d}{2}$, the Sobolev space $H^s(\mathbb{R}^d)$ is defined as
\[
H^s(\mathbb{R}^d) = \left\{ f \in L^2(\mathbb{R}^d) : \|f\|_{H^s(\mathbb{R}^d)}^2 = \int_{\mathbb{R}^d} (1 + |\xi|^2)^s |\hat{f}(\xi)|^2 \, d\xi < \infty \right\},
\]
with $\hat{f}$ the Fourier transform of $f$. For a bounded domain $\mathcal{D} \subset \mathbb{R}^d$, define
\[
H^s(\mathcal{D}) = \left\{ f \in L^2(\mathcal{D}) : \exists F \in H^s(\mathbb{R}^d),\ F|_{\mathcal{D}} = f \right\}, \quad
\|f\|_{H^s(\mathcal{D})} = \inf \left\{ \|F\|_{H^s(\mathbb{R}^d)} : F|_{\mathcal{D}} = f \right\}.
\]

Let $h_{n,\Omega}:= \sup\limits_{x\in\Omega}{\rm{dist}}(x,X_n)$ is fill distance.
Let $\kappa: \mathcal{D} \times \mathcal{D} \to \mathbb{R}$ be a covariance function whose RKHS $\mathcal{H}_\kappa$ satisfies $\mathcal{H}_\kappa \subset H^{s_n}(\mathcal{D})$ for some $s_n > \frac{d}{2}$ with \emph{equivalent norms}: there exists $C_N > 1$ such that
\[
C_N^{-1} \|g\|_{\mathcal{H}_\kappa} \leq \|g\|_{H^{s_n}(\mathcal{D})} \leq C_N \|g\|_{\mathcal{H}_\kappa}, \quad \forall g \in \mathcal{H}_\kappa.
\]
Let $K_\kappa = (\kappa(x_i, x_j))_{i,j=1}^n \in \mathbb{R}^{n \times n}$ be invertible, and let $\lambda \geq 0$ be the nugget parameter.
The following theorem combines this general misspecification analysis with the convergence criterion of the weighted Jacobi method.

\begin{thm}
	\label{thm:mis_convergence}
	Let $A \in \mathbb{R}^{n \times n}$ be symmetric with diagonal part $D$, and let $\lambda_{\min}, \lambda_{\max} > 0$ be the smallest and largest eigenvalues of $D^{-1}A$. Let $f: \Omega \to \mathbb{R}$ be the true optimal relaxation parameter function, observed at points $X_n = \{x_1,\dots,x_n\} \subset \Omega$ with quasi-uniform fill distance $h_{n,\Omega} \lesssim n^{-1/d}$. Let $\kappa: \Omega \times \Omega \to \mathbb{R}$ be a misspecified covariance function whose RKHS $\mathcal{H}_\kappa$ satisfies $\mathcal{H}_\kappa \subset H^{s_N}(\Omega)$ $(s_N > d/2)$ with equivalent norms: there exists $C_N > 1$ such that
	\[
	C_N^{-1} \|g\|_{\mathcal{H}_\kappa} \leq \|g\|_{H^{s_N}(\Omega)} \leq C_N \|g\|_{\mathcal{H}_\kappa}, \quad \forall g \in \mathcal{H}_\kappa.
	\]
	Let $K_\kappa = (\kappa(x_i,x_j))_{i,j=1}^n$ be invertible, and let $\eta \geq 0$ be the nugget parameter. The GPR predicted function is denoted by $f^*(x)$. When using the weighted Jacobi iteration, convergence is guaranteed under the following sufficient conditions:
	
	\vspace{0.5em}
	\noindent\textbf{Part I}.  
	Suppose $f \in H^{s_0}(\Omega)$ with $s_0 > d/2$, and there exists $f_N \in \mathcal{H}_\kappa$ such that $f_N|_{X_n} = f|_{X_n}$. There exists $h_{\lfloor s_0 \rfloor, \lfloor s_N \rfloor, d, \Omega} > 0$ such that if $h_{n,\Omega} \leq h_{\lfloor s_0 \rfloor, \lfloor s_N \rfloor, d, \Omega}$, then
	\[\begin{aligned}
		\|f - f^*\|_{L^\infty(\Omega)} 
		&\leq C_{s_0,s_N,d,\Omega} C_N^2 \Big[(h_{n,\Omega}^{s_0} + h_{n,\Omega}^{d/2} \sqrt{\eta})(\|f\|_{H^{s_0}(\Omega)} + \|f_N\|_{H^{s_0}(\Omega)}) + (h_{n,\Omega}^{s_N} + h_{n,\Omega}^{d/2} \sqrt{\eta})\|f_N\|_{H^{s_N}(\Omega)}\Big] \\
		&\le \min\left\{ \frac{2\lambda_{\min}}{\lambda_{\max}(\lambda_{\min} + \lambda_{\max})}, \frac{2\lambda_{\max}}{\lambda_{\min}(\lambda_{\min} + \lambda_{\max})} \right\}.
	\end{aligned}\]
	If $\eta = 0$, the constant $C_{s_0,s_N,d,\Omega}$ can be replaced by the improved value $C_{\lfloor s_0 \rfloor, \lfloor s_N \rfloor, d, \Omega}$.
	
	\vspace{0.5em}
	\noindent\textbf{Part II}.  
	Suppose $f \in L^\infty(\Omega)$ has well-defined pointwise values. There exists $h_{\lfloor s_N \rfloor, d, \Omega} > 0$ such that if $h_{n,\Omega} \leq h_{\lfloor s_N \rfloor, d, \Omega}$, then for any $f_N \in \mathcal{H}_\kappa$,
	\[\begin{aligned}
		\|f - f^*\|_{L^\infty(\Omega)} 
		&\leq C_{s_N,d,\Omega} C_N \sqrt{n} \left[ \frac{h_{n,\Omega}^{s_N - d/2}}{\sqrt{\sigma_{\min}(K_\kappa) + \eta}} + 1 \right] \|f_N - f\|_{L^\infty(\Omega)} + C_{s_N,d,\Omega} C_N^2 (h_{n,\Omega}^{s_N - d/2} + \sqrt{\eta})\|f_N\|_{H^{s_N}(\Omega)}\\
		&\le \min\left\{ \frac{2\lambda_{\min}}{\lambda_{\max}(\lambda_{\min} + \lambda_{\max})}, \frac{2\lambda_{\max}}{\lambda_{\min}(\lambda_{\min} + \lambda_{\max})} \right\}.
	\end{aligned}
	\]
	If $\eta = 0$, the constant $C_{s_N,d,\Omega}$ can be replaced by $C_{\lfloor s_N \rfloor, d, \Omega}$.

\end{thm}

{ 
\subsection{Complexity analysis}

The computational complexity of all proposed methods is considered, which mainly includes the parameter prediction by GPR and the weighted Jacobi iteration.

The parameter prediction using GPR involves computations for both training data and prediction data. Assume the training dataset for GPR consists of $M$ samples, which satisfies $M \ll N$, where $N$ is the size of the linear system to be solved.
The training process requires matrix inversion, which has a time complexity of  $\mathcal{O}(M^3)$ ~\cite{C-C-2006}. When predicting the optimal parameter $\omega^*$, it has a time complexity of $\mathcal{O}(M^2)$ due to the need to compute the kernel vector and its product with the inverse covariance matrix. Hence, the overall complexity of the GPR-based parameter prediction is $\mathcal{O}(M^3 + M^2)$, which is dominated by the training cost $\mathcal{O}(M^3)$.

Each iteration of the weighted Jacobi iteration mainly involves computing the matrix-vector product $Ax$.
Consequently, the time complexity of one WJI iteration is $\mathcal{O}(N^2)$~\cite{Y-2003}.
If the algorithm requires $k$ iterations to converge, the total complexity of the iterative process is $\mathcal{O}(kN^2)$.

Therefore, the overall computational complexity of the WJI based on GPR is 
$\mathcal{O}(M^3 + kN^2)$.

}
\section{Numerical Results}
\label{Experiments}

In this section, we apply both the { WJI} based on GPR and the Jacobi iteration (JI) to solve the linear system $Ax = b$. 
We demonstrate that the former outperforms the latter in terms of both the number of iteration steps (denoted as  ``Iter") and computational time in seconds (denoted as ``CPU"). { The CPU time represents the total time for parameter prediction and iterative solving. Since the same type of PDEs can reuse the same training set for all matrix sizes, the time spent generating the training set is not taken into account.}

In Example \ref{exm:2d-laplace-neumann} $\sim$ Example \ref{exm:3d-strong_convection_diffusion}, the tests are started from the initial vector is $\frac{e_1}{\Vert e_1\Vert_2}$, where $e_1 = (1,1,\dots,1)^T$, and terminated once the current relative residual, defined by
$$
\text{RRES} = \frac{\Vert r_k\Vert_2}{\Vert r_0\Vert_2},
$$
satisfies $\text{RRES} < 10^{-6}$, or once the number of iteration steps exceeds 500000, where $r_k =D^{-1} (b-Ax_k)$.

All experiments are carried out using MATLAB (version R2022a) on a 
on a Mac laptop with Apple M1 core and 8G memory.

We determine the relatively quasi-optimal $\omega$ by traversing method for small-scale matrices and then apply the GPR method to predict it optimal value. 
The small-scale training set is selected with $n \in [5,50]$ where the samples are spaced by an interval of 5.
The quasi-optimal parameter $w$ is obtained by traversing the interval $[0.5,~2]$, with a step size of 0.001. 

In the process of GPR, { we utilize three different kernel functions to capture the varying potential behaviors of optimal parameters across different matrix scales. We choose the Gaussian kernel $\mathcal{K}_1$ as it is one of the most commonly used kernel functions, and the Periodic kernel $\mathcal{K}_2$ is selected to capture any potential periodicity in the optimal parameter values. Additionally, we include an additive kernel $\mathcal{K}_3$, which is the sum of the Gaussian and Periodic kernels, to leverage the strengths of both kernels and provide a more flexible modeling approach. Specifically, these kernels are defined as follows:

$$
\mathcal{K}_1(x_i,x_j) = \sigma_f^2 \frac{\exp(-\Vert x_i-x_j\Vert_2^2)}{2\sigma_l^2},~~\mathcal{K}_2(x_i,x_j) = \sigma_f^2  \exp \left( \frac{-2\sin^2(\frac{ \pi \Vert x_i-x_j\Vert_2 }{p})}{\sigma_l^2} \right), 
$$
$$
\mathcal{K}_3(x_i,x_j) = \mathcal{K}_1(x_i,x_j)+\mathcal{K}_2(x_i,x_j), 
$$
where $\sigma_f$, $\sigma_l$, $p$ are hyperparameters.

The hyperparameters in the kernel function are selected by maximizing the log-likelihood function:

\begin{equation}
\log p(Y|X, \boldsymbol{\theta}) = -\frac{1}{2} Y^T (K_{\boldsymbol{\theta}} + \sigma_n^2 I)^{-1} Y - \frac{1}{2}\log|K_{\boldsymbol{\theta}} + \sigma_n^2 I| - \frac{n}{2}\log 2\pi,
\end{equation}
where $\boldsymbol{\theta}$ represents the hyperparameters of the kernel function, $K_{\boldsymbol{\theta}}$ is the kernel matrix computed using the training data and the kernel function with $\boldsymbol{\theta}$.

{  In the analysis of numerical results, "WJI+Gaussian" denotes parameter prediction for the WJI using GPR with the Gaussian kernel; similarly for "WJI+Periodic" and "WJI+additive kernel".}
}

\begin{exm}
	\label{exm:2d-laplace-neumann}
	Consider the two-dimensional Laplace equation with Neumann boundary conditions:
	\begin{equation}
		\label{eq:2d-laplace}
		-\left( \frac{\partial^2 u}{\partial x^2} + \frac{\partial^2 u}{\partial y^2} \right) = 0, \quad (x, y) \in [0, L] \times [0, L],
	\end{equation}
	with Neumann boundary conditions:
	\[
	\frac{\partial u}{\partial n} = 0 \quad \text{on } \partial \Omega,
	\]
	where \( \partial \Omega \) is the boundary of the domain \( \Omega = [0, L] \times [0, L] \).
\end{exm}

We use the centered difference method to discretize the Laplace equation (\ref{eq:2d-laplace}) and obtain the linear system \( A u = b \). The coefficient matrix \( A \) is a sparse matrix of size \( (N_x-2)(N_y-2) \times (N_x-2)(N_y-2) \), where \( N_x \) and \( N_y \) are the number of grid points in the \( x \) and \( y \) directions, respectively. The matrix \( A \) is defined as:
\[
A = D_x \otimes I + I \otimes D_y,
\]
where   \( D_x \) and \( D_y \) are tridiagonal matrices representing the discretization of the second derivatives in the \( x \) and \( y \) directions, respectively, \( I \) is the identity matrix, \( \otimes \) denotes the Kronecker product.
The right-hand side vector \( b \) is defined as $b = (0,0,\dots,0)$, where the elements of \( b \) are determined by the Neumann boundary conditions.

Table \ref{Tab:2d_Laplace_Neumann} shows Iter and the CPU time of WJI with three different kernels and JI. The WJI with three different kernels is significantly superior to JI in terms of Iter and CPU time. 
{ 
As the matrix size $n$ increases, the iteration count of JI grows substantially, whereas the WJI based on GPR maintains an almost constant Iter of about 3000 across all tested sizes, leading to significantly lower CPU time. These results demonstrate that parameter predicted by GPR  enhance the convergence performance of the WJI.
}

Figure \ref{Figure:2d_laplace_Neumann} displays the curves of Iter versus $n$, $\log_{10} \text{RRES}$ versus iteration steps, which can be seen in Figure \ref{Figure:2d_laplace_Neumann} (a) and (b). {  From these, it is clear that using the parameter predicted by GPR can significantly accelerate the convergence rate of the WJI. Figure \ref{Figure:2d_laplace_Neumann} (c) illustrates the predicted $\omega$ values versus grid size $n$. Here we select 10 training data points and use them to predict $\omega$ for larger values of $n$. It can be observed that the $\omega$ values predicted by the three kernel functions eventually stabilize around 1.065. In fact, while more training data generally leads to more accurate predictions, 10 points turn out to be sufficient for this problem.}  Figure \ref{Figure:2d_laplace_Neumann} (d) shows the curves of predicted variance of three kernels at different matrix sizes $n$ and demonstrates that all three kernel functions maintain prediction variances below 0.4, indicating that our parameter predictions approximate the optimal parameters with satisfactory accuracy.
While periodic kernels achieved the lowest prediction variance among the tested kernels, other kernels demonstrated superior performance in terms of Iter and CPU times. This phenomenon can be attributed to the influence of $\Vert f\Vert_\mathcal{H}$ on the effectiveness of optimal parameter prediction.


\begin{table}[htbp!]
	\caption{Iter and CPU of the JI and WJI based on GPR for { 2D Laplace equation with Neumann boundary conditions}}
	\label{Tab:2d_Laplace_Neumann}
	\centering
	\begin{tabular}{ccccccccc}
		\hline
		\multirow{2}{*}{Iteration method} &     & \multicolumn{7}{c}{$n$}           \\
		\cline{3-9} 
		&     & 60                   & 70   & 80    & 90    & 100   & 110   & 120 \\ 
		\hline
		\multirow{2}{*}{JI}           & Iter   & 5824 &  9503  & 13057 & 16159 & 21165 &  24586 &25166 \\
		\cline{2-9} 
		& CPU & 0.26634  & 0.54209  &  0.99525  &  1.6187  & 4.8451 &  6.7412 &   7.9442 \\  \hline
		\multirow{2}{*}{  WJI+Gaussian}         & Iter  &  2430  & 2630 & 2519 &  1656 & 2537 &  3585 & 2901 \\ 
		\cline{2-9} 
		& CPU & 0.1539 &   0.1052 &  0.23903 & 0.25884 & 0.76803 &   1.3255 & 1.1907 \\ 
		\hline
		\multirow{2}{*}{  WJI+Periodic}         & Iter  &  2332 &   2674 & 2524 & 1674 & 2554 & 3585 &2873 \\ 
		\cline{2-9} 
		& CPU &  0.12054 &  0.1881 & 0.24059 &  0.25146 & 0.77264 &  1.2685 & 1.1814 \\ 
		\hline
		\multirow{2}{*}{  WJI+additive kernel}         & Iter  & 2428 &  2622 & 2527 &  1645  & 2556 & 3581 &  2895 \\ 
		\cline{2-9} 
		& CPU &   0.12495 & 0.1795        &   0.23928 &  0.20878  & 0.78836 & 1.2706 & 1.1933 \\ 
		\hline
	\end{tabular}
\end{table}

\begin{figure}[h!]
	\begin{minipage}[t]{0.45\textwidth}
		\centering
		\includegraphics[width=1\textwidth]{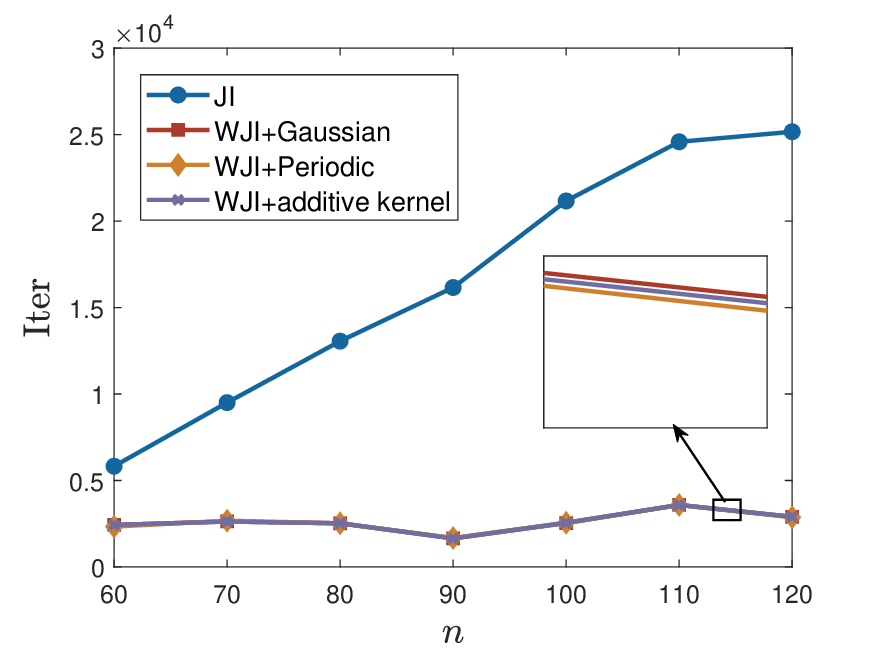}
		\caption*{(a) {  Iteration versus $n$}}
	\end{minipage}
	\begin{minipage}[t]{0.45\textwidth}
		\centering
		\includegraphics[width=1\textwidth]{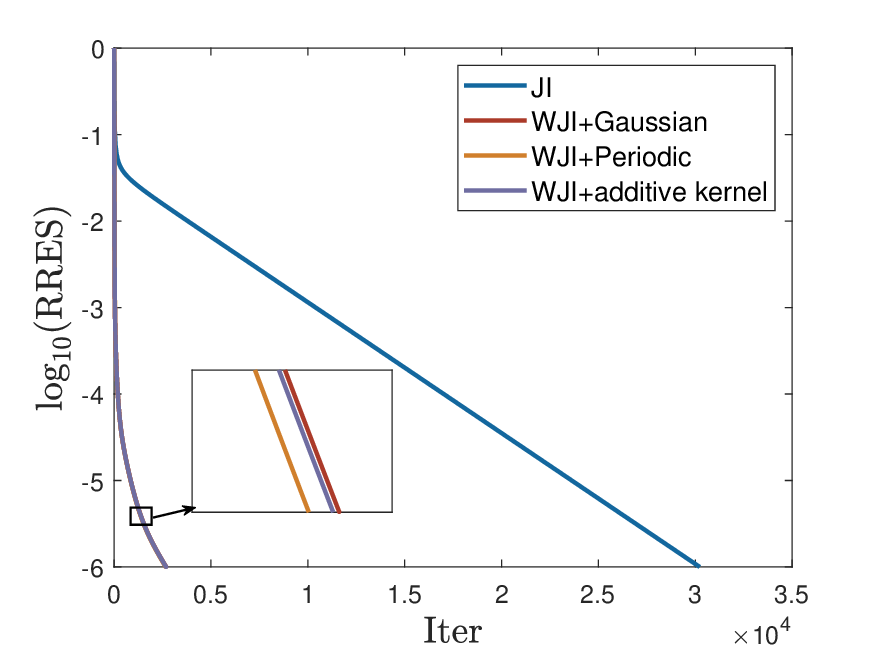}
		\caption*{(b) {  $\log_{10}\text{RRES}$ versus iteration}}
	\end{minipage}
	
	\begin{minipage}[t]{0.45\textwidth}
		\centering
		\includegraphics[width=1\textwidth]{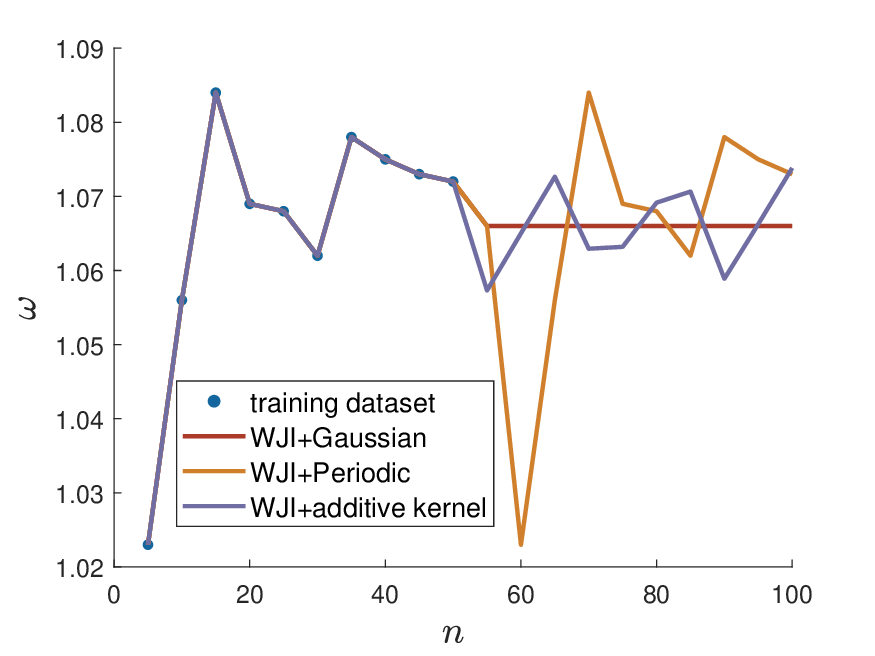}
		\caption*{(c) {  $n$ versus $\omega$}}
	\end{minipage}
	\begin{minipage}[t]{0.45\textwidth}
		\centering
		\includegraphics[width=1\textwidth]{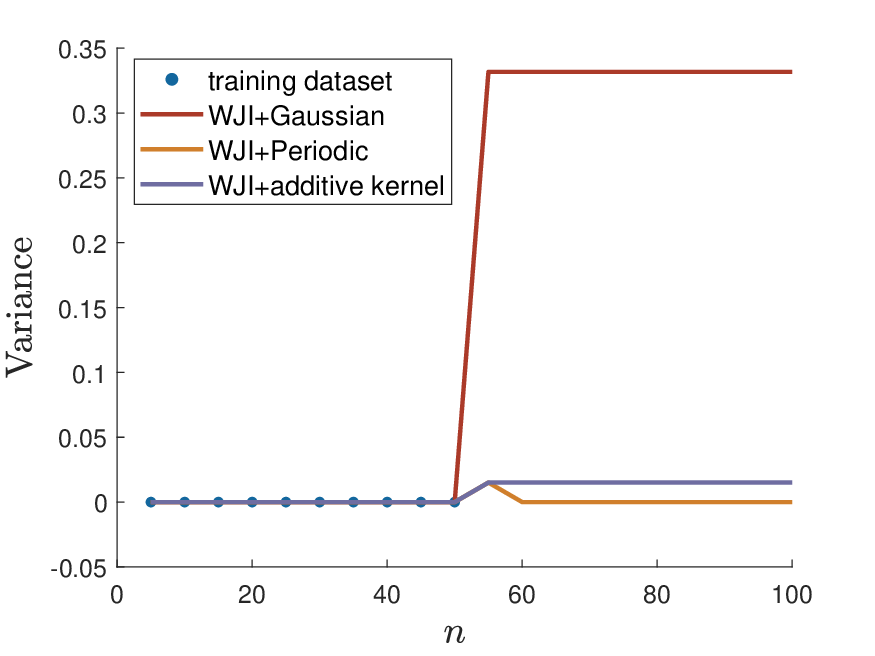}
		\caption*{(d) {  Variance versus $n$}}
	\end{minipage}
	
	\caption{{ Numerical results for the 2D Laplace equation with Neumann boundary conditions. (a): Curve of the number of iterations versus $n$; (b):Convergence curve of RRES when $n = 120$; (c): Values of $\omega$ for different $n$;(d): Variance Curves for different $n$.}}
	\label{Figure:2d_laplace_Neumann}
\end{figure}

\begin{exm}[\cite{J-S-Z-2022}]
	\label{exm:3d-convection_diffusion}
	Consider the three-dimensional convection-diffusion equation with Dirichlet boundary condition: 
	\begin{equation}
		\label{eq:3d-convection_diffusion}
		-(\frac{\partial^2 u}{\partial x^2} + \frac{\partial^2 u}{\partial y^2}+\frac{\partial^2 u}{\partial z^2})+(\frac{\partial u}{\partial x}+\frac{\partial u}{\partial y}+\frac{\partial u}{\partial z})  = f(x,y,z), 
	\end{equation}
	on the unit cube $\Omega = [0,1]\times [0,1]\times [0,1]$.
\end{exm}

We use the centered difference method to discretize the convective-diffusion (\ref{eq:3d-convection_diffusion}) and obtain the linear system $Ax = b$. The coefficient matrix is 
$$
A = T_x \otimes I \otimes I + I \otimes T_y \otimes I + I \otimes I \otimes T_z
$$
where $T_x$, $T_y$ and $T_z$ are tridiagonal matrices. $T_x = \text{tridiag}(t_2,t_1,t_3)$, $T_y = T_z =  \text{tridiag}(t_2,0,t_3)$, $t_1 = 6$, $t_2 = -1-r$, $t_3 = -1+r$, $r = \frac{1}{2n+2}$, $n$ is the degree
of freedom along each dimension. We choose $n = 60,~70,~80,~90,~100,~110,~120$. $x\in \mathbb{R}^{n^3}$ is the unknown vector of discretizing $u(x,y,z)$. $b \in \mathbb{R}^{n^3}$ is the discretization vector of $f(x,y,z)$ which is determined by choosing the exact solution $x_e = (1,1,\dots,1)^T$.

The numerical results of Example \ref{exm:3d-convection_diffusion} are listed in Table \ref{Tab:3d_convection_diffusion}. It can be seen that the WJI based on GPR significantly improves Iter and CPU time compared to JI, particularly for ($n \geq 100$). However, when $n = 60$, {  the WJI based on GPR} requires fewer iterations than JI but incurs slightly higher CPU time. This arises because the WJI introduces a scalar multiplication in its update step. This additional operation increases the per-iteration computational cost compared to JI. For smaller matrices $ n = 60$ , the reduction in iterations is insufficient to offset this overhead, leading to higher CPU time.
Figure \ref{Figure:3d_convection_diffusion} illustrates the convergence, variance prediction, and variance variation curves, visually confirming {  the rationality of the GPR-predicted parameter. Moreover, the results demonstrate that this approach successfully accelerates the convergence of WJI. We observe that as $n$ increases, the $\omega$ values predicted by the three kernel functions stabilize between 1.02 and 1.025, and their acceleration effects in WJI are comparable.}

\begin{table}[htbp!]
	\caption{Iter and CPU of the JI and WJI based on GPR for {  3D convection-diffusion equation with Dirichlet boundary condition} }
	\label{Tab:3d_convection_diffusion}
	\centering
	\begin{tabular}{ccccccccc}
		\hline
		\multirow{2}{*}{Iteration method} &     & \multicolumn{7}{c}{$n$}           \\
		\cline{3-9} 
		&     & 60    & 70   & 80 & 90  & 100 & 110 & 120 \\ 
		\hline
		\multirow{2}{*}{JI}           & Iter  &  6189 & 8353  & 10250 & 9165 &15734 & 19809 &  22056\\
		\cline{2-9} 
		& CPU & 44.058   & 112.50    &   188.43   & 233.77  &  547.15 &907.42 & 1291.5 \\  \hline
		\multirow{2}{*}{  WJI+Gaussian}         & Iter  & 4668   & 6005  & 7456   & 9008 &  10680 &12402 &14232 \\ 
		\cline{2-9} 
		& CPU &  48.872 &  90.389 & 189.46 &  280.37  & 435.43 &678.15 & 950.92 \\ 
		\hline
		\multirow{2}{*}{  WJI+Periodic}         & Iter  &  4705   &  5934 &  7441  &  8991  & 10659 &12402 &14344 \\ 
		\cline{2-9} 
		& CPU & 50.606 & 86.599 & 159.74 & 291.21 & 435.30 & 699.51& 936.85 \\ 
		\hline
		\multirow{2}{*}{  WJI+additive kernel}         & Iter  & 4676  &  6014  & 7467  &  9022 & 10696 & 12421 &14253 \\ 
		\cline{2-9} 
		& CPU & 50.804 & 86.027 & 194.57 & 243.09 & 417.94 & 642.49 & 933.05 \\ 
		\hline
	\end{tabular}
\end{table}

\begin{figure}[h!]
	\begin{minipage}[t]{0.43\textwidth}
		\centering
		\includegraphics[width=1\textwidth]{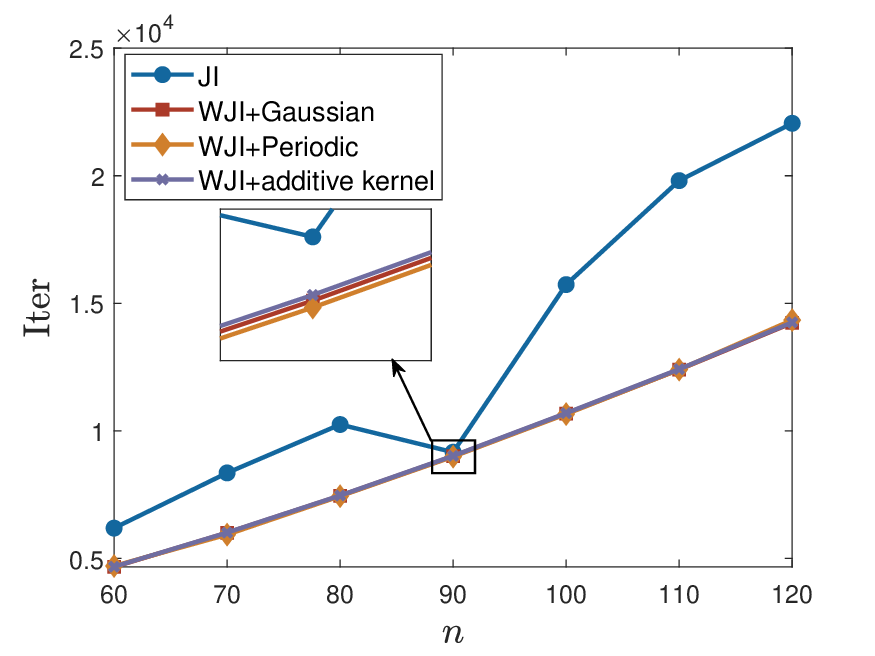}
		\caption*{(a){  Iteration versus $n$}}
	\end{minipage}
	\begin{minipage}[t]{0.43\textwidth}
		\centering
		\includegraphics[width=1\textwidth]{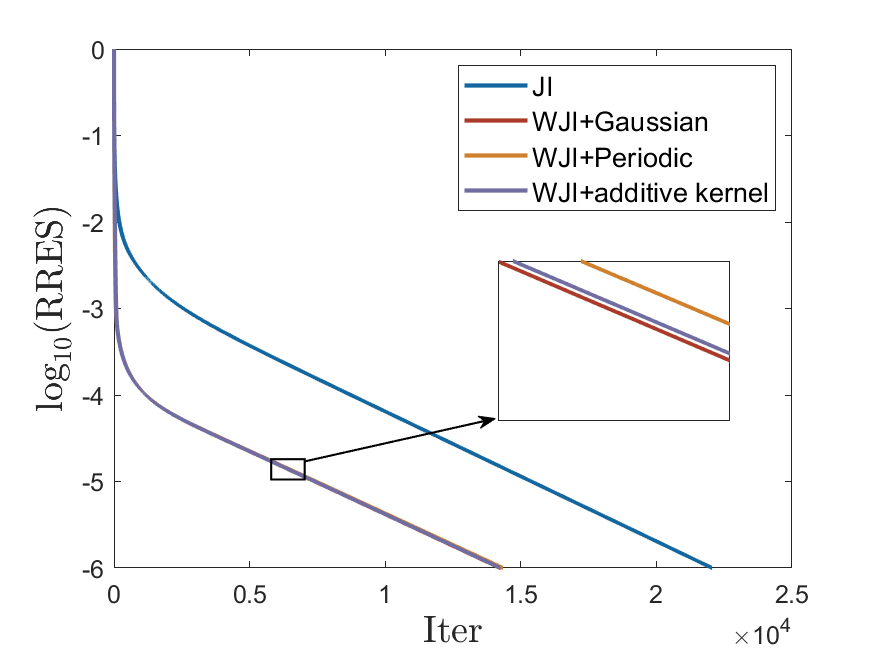}
		\caption*{(b){  $\log_{10}\text{RRES}$ versus iteration}}
	\end{minipage}\\
	\begin{minipage}[t]{0.43\textwidth}
		\centering
		\includegraphics[width=1\textwidth]{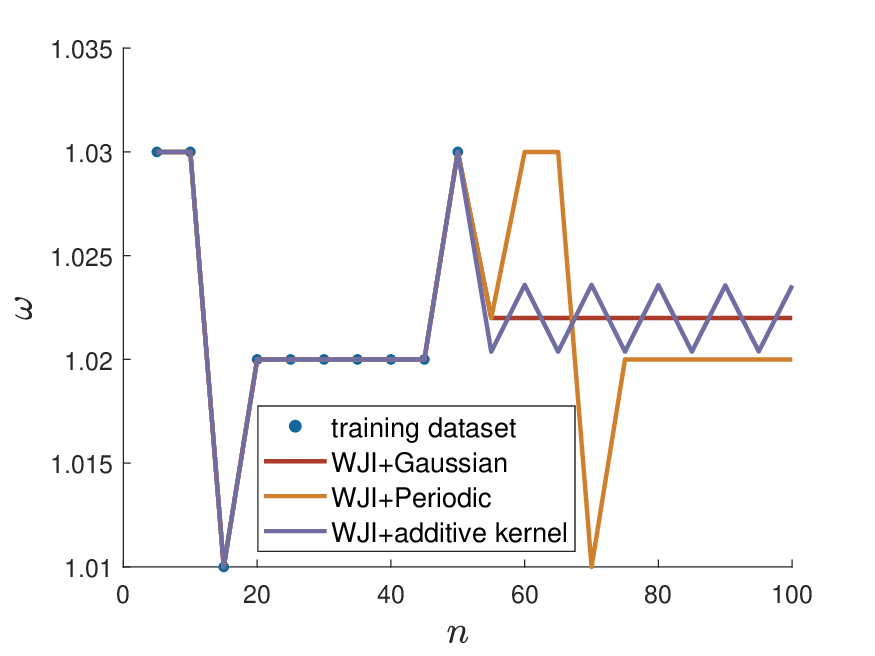}
		\caption*{(c) {  $n$ versus $\omega$}}
	\end{minipage}
	\begin{minipage}[t]{0.43\textwidth}
		\centering
		\includegraphics[width=1\textwidth]{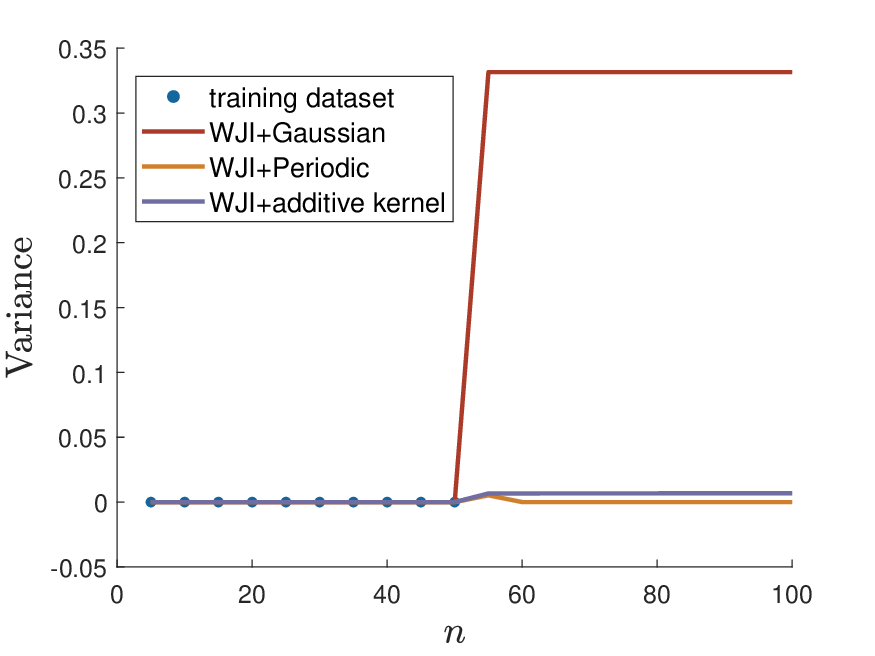}
		\caption*{(d){  Variance versus $n$}}
	\end{minipage}
	\caption{ {Numerical results for the 3D convection diffusion equation with Dirichlet  boundary conditions. (a): Curve of the number of iterations versus $n$; (b):Convergence curve of RRES when $n = 120$; (c): Values of $\omega$ for different $n$;(d): Variance Curves for different $n$.}}
	\label{Figure:3d_convection_diffusion}
\end{figure}

\begin{exm}
	\label{exm:3d-strong_convection_diffusion} Consider the three-dimensional convection-diffusion equation with
	{   a small diffusion coefficient and Dirichlet boundary condition}:
	\begin{equation}
		\label{eq:3d-strong_convection_diffusion}
		-\epsilon\left(\frac{\partial^2 u}{\partial x^2} + \frac{\partial^2 u}{\partial y^2}+\frac{\partial^2 u}{\partial z^2}\right)+\beta\left(\frac{\partial u}{\partial x}+\frac{\partial u}{\partial y}+\frac{\partial u}{\partial z}\right)  = f(x,y,z), 
	\end{equation}
	on the unit cube $\Omega = [0,1]\times [0,1]\times [0,1]$, where $\epsilon = 0.01$ and $\beta = 3$.
\end{exm}

We use the finite difference method with centered differencing for diffusion and backward differencing for convection to discretize (\ref{eq:3d-strong_convection_diffusion}) and obtain the linear system $Ax = b$. The coefficient matrix is constructed as
\[
A = A_x \otimes I \otimes I + I \otimes A_y \otimes I + I \otimes I \otimes A_z,
\]
where the one-dimensional discrete operators are given by:
\[
A_x = A_y = A_z = -\epsilon\cdot\text{tridiag}(-1, 2, -1)/h^2 + \beta\cdot\text{tridiag}(-1, 1, 0)/(h),
\]
and $h = 1/(n+1)$ is the grid spacing. Here $n$ is the number of interior grid points along each dimension. We choose $n = 60,~70,~ 80,~90,~100,~110,~120$. The vector $x \in \mathbb{R}^{n^3}$ is the unknown vector obtained by discretizing $u(x,y,z)$ in lexicographical order. The right-hand side $b \in \mathbb{R}^{n^3}$ is the discretization of $f(x,y,z)$, determined by the exact solution
\[
u_e(x,y,z) = \sin(\pi x)\sin(\pi y)\sin(\pi z),
\]
so that $f(x,y,z)$ is computed by applying the differential operator in (\ref{eq:3d-strong_convection_diffusion}) to $u_e$. The Dirichlet boundary values are imposed from $u_e$.

\begin{table}[htbp!]
	\caption{Iter and CPU of the JI and WJI based on GPR for { the 3D convection-diffusion equation with a small diffusion coefficient and Dirichlet boundary condition} }
	\label{Tab:3d_strong_convection_diffusion}
	\centering
	\begin{tabular}{ccccccccc}
		\hline
		\multirow{2}{*}{Iteration method} &     & \multicolumn{7}{c}{$n$}           \\
		\cline{3-9} 
		&     & 60    & 70   & 80 & 90  & 100 & 110 & 120 \\ 
		\hline
		\multirow{2}{*}{JI}           & Iter  & 300 & 397 & /& /& / & / & / \\
		\cline{2-9} 
		& CPU & 1.6303 &  3.1958  & / &/ & / & / & / \\  \hline
		\multirow{2}{*}{{ WJI+Gaussian}}         & Iter  & 223    & 283 &357& 451 & 571 &593 & 774 \\ 
		\cline{2-9} 
		& CPU & 1.4581 &2.8086  & 5.3823 &9.0178&16.199& 21.998 & 37.167 \\ 
		\hline
		\multirow{2}{*}{{ WJI+Periodic}}         & Iter  &  241   &  302 & 372 & 430 &509& 468 &548  \\ 
		\cline{2-9} 
		& CPU & 1.5621 & 3.0272 & 5.4891 &8.5613& 14.012 & 17.257 & 26.437  \\ 
		\hline
		\multirow{2}{*}{{ WJI+additive kernel}}         & Iter  &  225 &284  & 349 & 459 & 737 & /  & / \\ 
		\cline{2-9} 
		& CPU & 1.4530 & 2.8615 & 5.1777 & 9.1335&  19.94 &  / & /  \\ 
		\hline
	\end{tabular}
\end{table}

\begin{figure}[h!]
\begin{minipage}[t]{0.43\textwidth}
	\centering
	\includegraphics[width=1\textwidth]{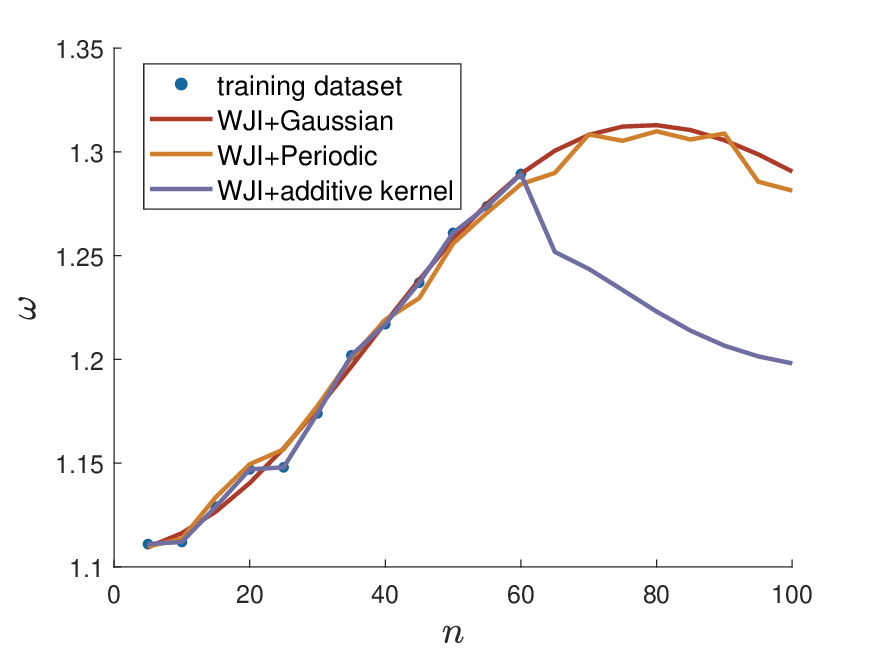}
	\caption*{(a) {  $n$ versus $\omega$}}
\end{minipage}
\begin{minipage}[t]{0.43\textwidth}
	\centering
	\includegraphics[width=1\textwidth]{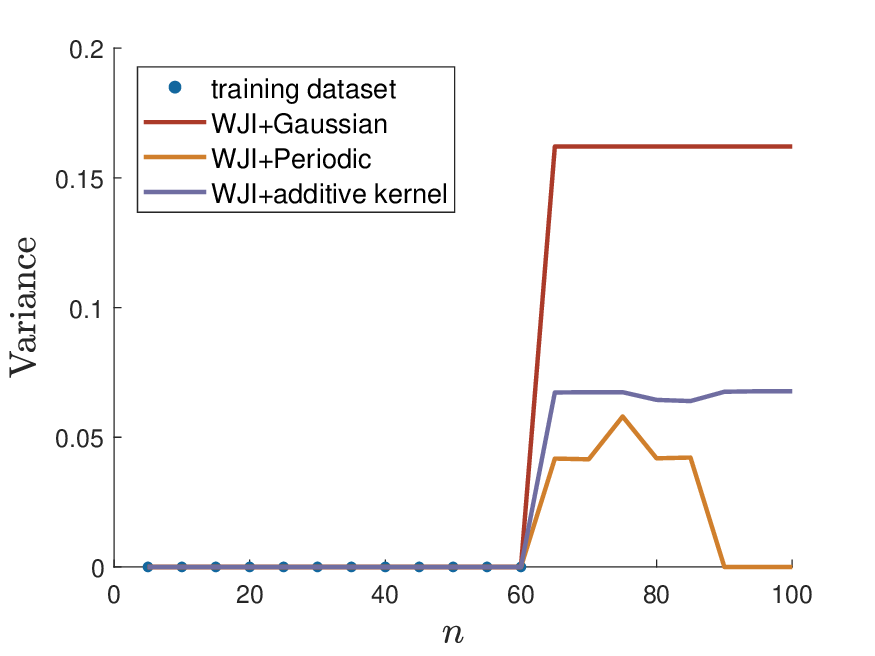}
	\caption*{(b) {  Variance versus $n$}}
\end{minipage}

\caption{  Prediction curves of parameter $\omega$ and corresponding variance curves under different $n$ in { the 3D convection-diffusion equation with a small diffusion coefficient and Dirichlet boundary condition}.}
\label{Figure:3d_strong_convection_diffusion}
\end{figure}

Table \ref{Tab:3d_strong_convection_diffusion} shows the numerical results for Example \ref{exm:3d-strong_convection_diffusion}. For $n \ge 80$, the JI fails to converge within 500000 iterations, whereas the WJI based on GPR converges within 1000 iterations. Among the kernels tested, the periodic kernel provides the best acceleration and exhibits the smallest prediction variance, as shown in Figure \ref{Figure:3d_strong_convection_diffusion} (b).
Figure \ref{Figure:3d_strong_convection_diffusion} (a) shows the predicted $\omega$ for different $n$. It can be observed that the predictions from the periodic kernel and the Gaussian kernel are close. However, they exhibit significant differences in their variances. This discrepancy may be attributed to the coefficient $C$ which is determined by the true function of $\omega$. Although the constant $C$ remains difficult to determine in practice, minimizing the predictive variance provides an effective criterion for improving prediction accuracy. {  From these results, we know that parameter predicted by GPR can accelerate the convergence of WJI, while the acceleration effect depends on the choice of kernel function, which can be evaluated by their prediction errors.}

\section{Conclusion}
\label{Conclusion}

We introduce a novel proof method for establishing the error bound of GPR using function decomposition. Additionally, we propose a WJI within the GPR parameter prediction framework and provide its convergence analysis, which is generalizable and can be combined with established GPR error bounds to enhance the theoretical framework. Numerical experiments apply GPR to WJI parameter selection with three different kernels, demonstrating that GPR-predicted parameters enhance JI convergence and validate the effectiveness of GPR for parameter prediction. In the future, we will investigate effective solution strategies for nonlinear problems and reliable parameter estimation methods and explore other efficient machine learning techniques to give proper parameters for algorithms, accompanied by corresponding error analysis. {  
Moreover, we will extend the proposed GPR-based parameter prediction framework to SOR iterations, where the relaxation parameter is embedded in the matrix inverse and cannot be explicitly extracted, and analyze how the error in predicting this parameter affects the convergence rate through its impact on the spectral radius.
Furthermore, we will investigate GPR-based prediction of parameters in the preconditioner of Krylov subspace methods. In the future, we will further study strategies to accelerate the weighted Jacobi iteration using parallel and mixed-precision approaches, and consider how to combine parameter prediction methods with these acceleration techniques.}

\section*{Acknowledgements}
	This work was supported by the National Natural Science Foundation of China under Grant (2023YFB3001604) and {  the Postgraduate Scientific Research Innovation Project of Hunan Province (CX20250937).}
	{ The authors would like to thank the associated editor and the anonymous referees for their constructive comments and suggestions, which substantially improved the paper.}
	
\section*{Declarations}
	\textbf{Author Contributions}  Tiantian Sun writes the original draft and performs all numerical experiments. Juan Zhang provides the core idea and critical revisions to the manuscript. All authors agree with the publication of the article.

	 \textbf{Conflict of interest}~The authors declare that they have no Conflict of interest.




\begin{thebibliography}{99}

\bibitem{A-A-S-2022}
A. Capone, A. Lederer, S. Hirche.
Gaussian process uniform error bounds with unknown hyperparameters for safety-critical applications{ ,}
{\it International Conference on Machine Learning}, 
2022,162: 2609-2624.

\bibitem{A-J-S-2019}
A.Lederer, J.Umlauft, S.Hirche.
Uniform error bounds for Gaussian process regression with application to safe control,
{\it Advances in Neural Information Processing Systems},
2019, 657-667.

\bibitem{C-C-2006}
C.K.I Williams, C.E.Rasmussen.
Gaussian processes for machine learning, Cambridge, MA: MIT press, 2006.


\bibitem{D-R-2025} 
D. Sanz-Alonso, R. Yang. 
Gaussian process regression under computational and epistemic misspecification, 
{\it SIAM Journal on Numerical Analysis}, 2025, 63(2): 495-519

{  \bibitem{Y-1954}
D. Young. 
Iterative methods for solving partial difference equations of elliptic type,
{\it Transactions of the American Mathematical Society}, 
1954, 76(1): 92-111.}


\bibitem{F-D-S-M-2013}
F.Dondelinger, D.Husmeier, S.Rogers, M.Filippone.
ODE parameter inference using adaptive gradient matching with Gaussian processes,
{\it Artificial intelligence and statistics, PMLR},
2013, 31: 216-228, 


\bibitem{K-A-M-T-D-2022}
K.Hashimoto, A.Saoud b, M.Kishida, T.Ushio, D.V.Dimarogonas.
Learning-based symbolic abstractions for nonlinear control systems, 
{\it Automatica}, 
2022, 146: 110646.

\bibitem{J-S-Z-2022}
K.Jiang, X.Su, J.Zhang.
A general alternating-direction implicit framework with Gaussian process regression parameter prediction for large sparse linear systems,
{\it SIAM Journal on Scientific Computing}, 
2022, 44(4): A1960-A1988.


\bibitem{K-J-Q-2023}
K.Jiang, J.Zhang, Q.Zhou.
Multitask kernel-learning parameter prediction method for solving time-dependent linear systems,
{\it CSIAM Transactions on Applied Mathematics}, 
2023, 4(4): 672-695.

{ 
\bibitem{L-Q-K-H-2025}
L. Liang, Q. Pang, K. C. Toh, H. Yang. 
Nesterov's Accelerated Jacobi-Type Methods for Large-Scale Symmetric Positive Semidefinite Linear Systems, 
{\it SIAM Journal on Scientific Computing}, 
2025, 47(6): A3494-A3515.
}

\bibitem{M-D-C-2013}
M.P.Deisenroth, D.Fox, C.E.Rasmussen.
Gaussian processes for data-efficient learning in robotics and control,
{\it IEEE transactions on pattern analysis and machine intelligence}, 
2013, 37(2): 408-423.

\bibitem{M-P-D-B-2018}
M. Kanagawa, P. Hennig, D. Sejdinovic, B. K. Sriperumbudur{ .} Gaussian processes and kernel
methods: A review on connections and equivalences,
{\it arXiv},
2018. 

\bibitem{N-A-S-M-2012}
N.Srinivas, A. Krause, S.M.Kakade, M.W.Seeger.
Information-theoretic regret bounds for gaussian process optimization in the bandit setting,
{\it   IEEE transactions on information theory,}
2012, 58(5): 3250-3265.

\bibitem{P-A-R-M-2004}
P.G.Ciarlet, A.Iserles, R.V.Kohn, M.H.Wright.
Scattered data approximation{ ,}
Cambridge, Cambridge university press, 2004.

\bibitem{P-P-J-2016}
P.P.Pratapa, P.Suryanarayana, and J.E.Pask.
Anderson acceleration of the Jacobi iterative method: An efficient alternative to Krylov methods for large, sparse linear systems,
{\it Journal of Computational Physics}, 2016, 306: 43-54.


\bibitem{R-1999}
R.Schaback.
Native Hilbert spaces for radial basis functions I,
{\it New Developments in Approximation Theory }, 
1999, 132: 255-282.

\bibitem{R-W-2020}
R. Tuo, W. Wang. 
Kriging prediction with isotropic Mat\'ern correlations: Robustness and experimental designs,
{\it Journal of Machine Learning Research}, 
2020, 21(187): 1-38.

\bibitem{R-M-L-A-D-2014}
R.D\"{u}richen, M.A.F.Pimentel, L.Clifton, A.Schweikard, D.A.Clifton.
Multitask Gaussian processes for multivariate physiological time-series analysis,
{\it IEEE Transactions on Biomedical Engineering}, 
2014, 62(1): 314-322.

\bibitem{R-L-M-2025}
R.Reed, L.Laurenti, M.Lahijanian.
Error bounds for Gaussian process regression under bounded support noise with applications to safety certification,
{\it Proceedings of the AAAI Conference on Artificial Intelligence},
2025, 39(19): 20157-20164.

\bibitem{S-J-P-E-2023}
S.J.Petit, J.Bect, P.Feliot, E. Vazquez. 
Parameter selection in Gaussian process interpolation: an empirical study of selection criteria,
{\it SIAM/ASA Journal on Uncertainty Quantification}, 
2023, 11(4): 1308-1328.

\bibitem{T-T-F-2023}
T.Martin, T.B. Sch\"{o}n, F.Allg\"{o}wer.
Guarantees for data-driven control of nonlinear systems using semidefinite programming: A survey,
{\it Annual Reviews in Control},
2023, 56: 100911.

\bibitem{W-B-2022}
W.Wang, B.Y.Jing. 
Gaussian process regression: Optimality, robustness, and relationship with kernel ridge regression{ ,}
{\it Journal of Machine Learning Research}, 
2022, 23(193): 1-67.

\bibitem{W-P-W-2011}
W.Scott, P.Frazier, and W.Powell.
The correlated knowledge gradient for simulation optimization of continuous parameters using gaussian process regression,
{\it SIAM Journal on Optimization},
2011, 21(3): 996-1026.

\bibitem{W-R-C-2020}
W.Wang, R.Tuo, C.F.Jeff Wu.
On prediction properties of kriging: Uniform error bounds and robustness,
{\it Journal of the American Statistical Association}, 
2020, 115(530):920-930.

\bibitem{X-R-2014}
{ 
X. Yang, R. Mittal. 
Acceleration of the Jacobi iterative method by factors exceeding 100 using scheduled relaxation,
{\it Journal of Computational Physics}, 
2014, 274: 695-708.
}

\bibitem{Y-1954}
D. Young. 
Iterative methods for solving partial difference equations of elliptic type,
{\it Transactions of the American Mathematical Society}, 
1954, 76(1): 92-111.

\bibitem{Y-2003}
{ Y.Saad. 
Iterative Methods for Sparse Linear Systems, 2nd ed,
{\it Philadelphia, PA: Society for Industrial and Applied Mathematics}, 
2003.}

\bibitem{Y-B-H-A-2021}
Y.Chen, B.Hosseini, H.Owhadi, A.M.Stuart.
Solving and learning nonlinear PDEs with Gaussian processes,
{\it Journal of Computational Physics}, 
2021, 447: 110668.









\end{thebibliography}



\end{document}